\documentclass[12 pt]{article}%
\usepackage{amsmath, amsfonts, amsthm, color,latexsym}
\usepackage{amsmath, ulem}
\usepackage{amsfonts}
\usepackage{amssymb}
\usepackage{color, soul}
\usepackage[all]{xy}
\usepackage{graphicx}%
\providecommand{\U}[1]{\protect\rule{.1in}{.1in}}
\allowdisplaybreaks[4]
\newtheorem{theorem}{Theorem}[section]
\newtheorem{proposition}[theorem]{Proposition}
\newtheorem{corollary}[theorem]{Corollary}
\newtheorem{example}[theorem]{Example}
\newtheorem{examples}[theorem]{Examples}
\newtheorem{remark}[theorem]{Remark}

\newtheorem{lemma}[theorem]{Lemma}
\newtheorem{final remark}[theorem]{Final Remark}
\newtheorem{definition}[theorem]{Definition}
\allowdisplaybreaks[4]

\begin{document}

\title{Representation of sequence classes by operator ideals}
\author{Geraldo Botelho\thanks{Supported by FAPEMIG grants RED-00133-21 and APQ-01853-23. } ~and  Ariel S. Santiago\thanks{Supported by a FAPEMIG scholarship\newline 2020 Mathematics Subject Classification: 46A45, 46B45, 47B10,  47L20.\newline Keywords: Banach spaces, sequence classes, operator ideals, summing operators. }}
\date{}
\maketitle

\begin{abstract}  It is well known that weakly $p$-summable sequences in a Banach space $E$ are associated to bounded operators from $\ell_{p^*}$ to  $E$, and unconditionally $p$-summable sequences in $E$ are associated to compact operators from $\ell_{p^*}$ to  $E$. Generalizing these results to a quite wide environment, we characterize the classes of Banach spaces-valued sequences that are associated to (or represented by) some Banach operator ideal. Using these characterizations, we decide, among all sequence classes that usually appear in the literature, which are represented by some Banach operator ideal and which are not. Moreover, to each class that is represented by some Banach operator ideal, we show an ideal that represents it. Illustrative examples and additional applications are provided.
\end{abstract}

\section{Introduction}

\noindent{\bf 1.1 Setting the scene.} Let ${\cal L}(E,F)$ denote the space of bounded linear operators from the Banach space $E$ to the Banach space $F$, let ${\cal K}(E,F)$ be its closed subspace formed by all compact operators and let ${\cal N}(E,F)$ be the space of nuclear operators endowed with the nuclear norm. For $1 \leq p < \infty$, by $\ell_p^w(E)$ we denote the space of weakly $p$-summable $E$-valued sequences endowed with its natural norm (see \cite[p.\,32]{diestel+jarchow+tonge} or Example \ref{ey8n}), and by $\ell_p^u(E)$ its closed subspace formed by all unconditionally $p$-summable sequences, that is, sequences $(x_j)_{j=1}^\infty$ in $\ell_p^w(E)$ such that $\lim\limits_n\|(x_j)_{j=n}^\infty\|_{w,p} = 0$. Also, by $\ell_1(E)$ we denote the space of absolutely summable $E$-valued sequences. As usual, $(e_j)_{j=1}^\infty$ denotes the sequence of canonical unit vectors of scalar-valued sequence spaces, and $p^*$ denotes the conjugate of $p \in (1, +\infty)$. It is well known that the following correspondences
$$u \in {\cal L}(c_0, E) \mapsto (u(e_j))_{j=1}^\infty \in \ell_1^w(E), $$
$$u \in {\cal K}(c_0, E) \mapsto (u(e_j))_{j=1}^\infty \in \ell_1^u(E),$$
$$u \in {\cal L}(\ell_{p^*}, E) \mapsto (u(e_j))_{j=1}^\infty \in \ell_p^w(E), ~~1 < p < \infty,$$
$$u \in {\cal K}(\ell_{p^*}, E) \mapsto (u(e_j))_{j=1}^\infty \in \ell_p^u(E), ~~1 < p < \infty,$$
$$u \in {\cal N}(c_0, E) \mapsto (u(e_j))_{j=1}^\infty \in \ell_1(E), $$
are isometric isomorphisms. This facts can be found, e.g, in \cite{df, diestel+jarchow+tonge}; for the last one see \cite[Example 2.30]{ryan}.  Needless to say, this possibility of regarding a vector-valued sequence as an operator is very fruitful in Banach space theory. These are not the only known cases: for $1 < p < \infty$, denoting the space of Cohen strongly $p$-summing $E$-valued sequences by $\ell_p\langle E \rangle$ (see Example \ref{ey8n}), in \cite[Corollary 3.7]{fourie+rontgen} it is proved that the same correspondence $u \in {\cal N}(\ell_{p^*},F) \mapsto (u(e_j))_{j=1}^\infty \in \ell_p\langle E \rangle$ is an isometric isomorphism.

Of course, it is not a coincidence that the correspondence is always $u \mapsto (u(e_j))_{j=1}^\infty$, that the underlying space of operators is always a component of a Banach operator ideal, namely, ${\cal L}, {\cal K}$ and ${\cal N}$, and that the domain space is a predual of the scalar component of the class.

There are several other classes of sequences that have been playing a central role in Banach space theory, such as absolutely $p$-summable sequences in the linear and nonlinear theory of absolutely summing operators, and almost unconditionally summable sequences in the geometry of Banach spaces (type, cotype) and in probability in Banach spaces.

An attempt to unify the study of all such spaces of sequences was made in \cite{botelho+campos} with the concept of {\it sequence classes}, which are rules $X$ that assign to each Banach space $E$ a Banach space $X(E)$ of $E$-valued sequences with certain properties (cf. Section 2). This approach has proved to be quite fruitful, see, e.g., \cite{achour, achour2, botelho+campos+2, davidson, ariellama, jamilsonrenato, jamiljoed, baweja, baianosmedit, baianosmfat}. The usefulness of the examples described above for the classes $\ell_p^w, \ell_p^u$,  $\ell_p\langle\cdot\rangle$ and $\ell_1(\cdot)$ make the following question quite natural: Given a sequence class $X$, are there a scalar-valued sequence space $\lambda$ and a Banach operator ideal $\cal I$ such that the correspondence
$$ u \in  \mathcal{I}(\lambda, E) \mapsto (u(e_j))_{j=1}^\infty \in X(E)$$
is an isometric isomorphism for every Banach space $E$? In this case we say that the sequence class $X$ is represented by $\cal I$; and a sequence class is ideal-representable if it is represented by some Banach operator ideal. In this paper we develop a theory of ideal-representable sequence classes such that: (i)  The aforementioned representations emerge as particular instances. (ii) For all sequence classes that usually appear in the literature, we decide which are ideal-representable and which are not. (iii) Useful characterizations of ideal-representable sequence classes are proved. (iv) For any ideal-representable class, we exhibit a Banach operator ideal that represents it. (v) Applications outside the environment of ideal-representability are provided.

Among the many results we prove and the several examples we give, some noteworthy facts appear; for example, the sequence class $E \mapsto \ell_p(E)$ of absolutely $p$-summable sequences is ideal-representable if and only if $p = 1$.

\medskip

\noindent{\bf 1.2 Overview of the paper.} In Section 2 we fix the notation and describe the basics of sequence classes, including the most popular examples. We begin Section 3 by studying the scalar-valued sequence spaces that shall fit our purposes. Then we define ideal-representable sequence classes, give the first examples and prove some results that will give, for instance, the first examples of non ideal-representable classes. Next we prove some characterizations of ideal-representable sequence classes which will be used several times along the paper. We finish Section 3 using the characterizations to complete the picture of the usual sequence classes that are ideal-representable or not. In Section 4 we give two applications of ideal-representability that are not related to ideal-representability at first glance. Namely, we improve an important result from \cite{junek+matos} and a very recent result from \cite{ariellama}.

\section{Preliminaries}

For Banach spaces $E$ and $F$ over $\mathbb{K} = \mathbb{R}$ or $\mathbb{C}$, $E^*$ denotes the topological dual of $E$, $B_E$ denotes the closed unit ball of $E$ and ${\cal L}(E;F)$ denotes the Banach space of bounded linear operators from $E$ to $F$ with the usual operator norm.
The symbol $E \stackrel{1}{\hookrightarrow} F$ means that $E$ is a linear subspace of $F$ and $\|\cdot\|_F \leq \|\cdot\|_E$ on $E$. We write $E \stackrel{1}{=}F $ if $E = F$ isometrically. The identity operator on $E$ is denoted by ${\rm id}_E$.  By $c_{00}(E)$ and $\ell_\infty(E)$ we denote the spaces of eventually null and bounded $E$-valued sequences. For $j \in \mathbb{N}$, set $e_j := (0, \ldots, 0,1, 0,0, \ldots)$, where $1$ appears at the $j$-th coordinate.

A {\it sequence class} is a rule that assigns, to each Banach space $E$, a Banach space $X(E)$ of $E$-valued sequences such that $c_{00}(E) \subseteq X(E) \stackrel{1}{\hookrightarrow} \ell_\infty(E)$ and $\|e_j\|_{X(\mathbb{K})} = 1$ for every $j \in \mathbb{N}$.

\begin{example}\label{ey8n}\rm Let $1 \leq p < \infty$. The following correspondences are sequence classes:\\
$\bullet$ $E \mapsto \ell_\infty(E)$ = bounded  sequences, $E \mapsto c_0(E)$ = norm null  sequences, $E \mapsto c(E)$ = convergent  sequences, $E \mapsto c_0^w(E)$ = weakly null sequences, all of them endowed with the supremum norm.\\
$\bullet$ $E \mapsto \ell_p(E)$ = absolutely $p$-summable sequences, with the norm $\|(x_j)_{j=1}^\infty\|_p = \left(\sum\limits_{j=1}^\infty\|x_j\|^p\right)^{1/p}$.\\
$\bullet$  $E \mapsto \ell_p^w(E)$ = weakly $p$-summable sequences, with the norm $\|(x_j)_{j=1}^\infty\|_{w,p} = \sup\limits_{\varphi \in B_{E^*}} \!\! \|(\varphi(x_j))_{j=1}^\infty\|_p$.\\ 
$\bullet$ $E \mapsto \ell_p^u(E)$ = unconditionally $p$-summable sequences (see \cite[8.2]{df}), with the norm $\|\cdot\|_{w,p}$.\\
$\bullet$ $E \mapsto {\rm Rad}(E)$ = almost unconditionally summable sequences (see \cite[Chapter 12]{diestel+jarchow+tonge}), with the norm $\|(x_j)_{j=1}^\infty\|_{\rm Rad}= \left\|\sum\limits_{j=1}^\infty r_jx_j \right\|_{L_2([0,1];E)}$, where $(r_j)_{j=1}^\infty$ is the sequence of Rademacher functions. \\
$\bullet$ $E \mapsto {\rm RAD}(E)$ = almost unconditionally bounded sequences (see \cite{vakhania}), with the norm\\ $\|(x_j)_{j=1}^\infty\|_{\rm RAD}=$ $\sup\limits_n \|(x_j)_{j=1}^n\|_{\rm Rad}$. \\
$\bullet$ $E \mapsto \ell_p\langle E \rangle$ = Cohen strongly $p$-summable sequences (see \cite{cohen}), with the norm
$$ \|(x_j)_{j=1}^\infty\|_{C,p} :=\sup\limits_{(\varphi_j)_{j=1}^\infty \in B_{\ell_{p^*}^w(E^*)}}  \left| \sum_{j=1}^\infty \varphi_j(x_j) \right|=\sup\limits_{(\varphi_j)_{j=1}^\infty \in B_{\ell_{p^*}^w(E^*)}}   \sum_{j=1}^\infty \left|\varphi_j(x_j) \right|.$$
$\bullet$ $E \mapsto \ell_p^{\rm mid}(E)$ = mid $p$-summable sequences (see \cite{botelho+campos+santos}), with the norm
$$ \|(x_j)_{j=1}^\infty\|_{{\rm mid},p}:= \sup_{(\varphi_n)_{n=1}^\infty \in B_{\ell_p^w(E^*)}} \left( \sum_{n=1}^{\infty} \sum_{j=1}^{\infty} |\varphi_n(x_j)|^p\right)^{1/p}.$$ 
\end{example}

When referring to a sequence class $X$, we shall adopt the following {\it modus vivendi}: at any risk of ambiguity we shall write $X(\cdot)$, otherwise we simply write $X$. For example, we write $\ell_p^w$ instead of $\ell_p^w(\cdot)$, but we write $\ell_p(\cdot)$ and not $\ell_p$.

By $\cal I$ we denote a Banach operator ideal endowed with a norm $\|\cdot\|_{\cal I}$. For the theory of Banach operator ideals we refer to \cite{df, pietsch}.

Let $X,Y$ be sequence classes. According to \cite{botelho+campos}, a linear operator $u \in {\cal L}(E;F)$ is said to be $(X;Y)$-summing, in symbols $u \in \Pi_{X;Y}(E;F)$, if $(u(x_j))_{j=1}^\infty \in Y(F)$ whenever $(x_j)_{j= 1}^\infty \in X(E)$. In this case, the induced linear operator
$$\widehat{u} \colon X(E) \longrightarrow Y(F)~,~\widehat{u}((x_j)_{j= 1}^\infty) =  (u(x_j))_{j=1}^\infty,$$
is bounded and the expression $\|u\|_{\Pi_{X;Y}} := \|\widehat{u}\|$ makes $\Pi_{X;Y}(E;F)$ a Banach space. A sequence class $X$ is {\it linearly stable} if $\Pi_{X;X}(E;F) \stackrel{1}{=} {\cal L}(E;F)$ for all Banach spaces $E$ and $F$. All sequence classes listed in Example \ref{ey8n}, as well as the other ones we shall work with in this paper, are linearly stable.

Under the two conditions below, $\Pi_{X;Y}$ is a Banach operator ideal (see \cite[Theorem 3.6]{botelho+campos}):\\
$\bullet$ $X$ and $Y$ are linearly stable.\\
$\bullet$ $(\lambda_j)_{j=1}^\infty\in Y(\mathbb{K})$ and $\|(\lambda_j)_{j=1}^\infty\|_{Y(\mathbb{K})} \leq \|(\lambda_j)_{j=1}^\infty\|_{X(\mathbb{K})}$ whenever $(\lambda_j)_{j= 1}^\infty \in X(\mathbb{K})$.

Of course we are interested in the case that $\Pi_{X;Y}$ is a Banach operator ideal, so, whenever we work with $\Pi_{X;Y}$, we will always suppose that the two conditions above hold. 

\section{Ideal-representable sequence classes}

The main results and examples of the paper are presented in this section. As is clear in the Introduction, for a class $X$ to be represented by an operator ideal, it is necessary that its scalar component $X(\mathbb{K})$ is the dual of a sequence space. We begin by studying these scalar-valued sequence spaces.

\begin{definition}\rm A \textit{scalar sequence space } is a Banach space $\lambda$ formed by scalar-valued sequences, that is, $\lambda \subseteq \mathbb{K}^\mathbb{N}$, endowed with the usual coordinatewise algebraic operations, satisfying the following conditions:

(i) $c_{00} \subseteq \lambda \stackrel{1}{\hookrightarrow} \ell_\infty$.

(ii) $(e_j)_{j=1}^\infty$ is a  Schauder basis for $\lambda$.

(iii)  $\|e_j\|_{\lambda}=1$ for every $j\in\mathbb{N}$.\\
We also define
$$ \lambda_*:=\{ (\varphi(e_j))_{j=1}^\infty:\varphi\in \lambda^* \} \subseteq \mathbb{K}^\mathbb{N}.$$
Of course, this construction is closely related to the K\"othe dual of a space of  scalar-valued sequences. This connection will become even clearer in Lemma \ref{C4S1L2}.

\end{definition}

As usual, for $1 \leq p < \infty$, by $p^* \in (1, \infty]$ we denote the conjugate of $p$, that is, $\frac1p + \frac{1}{p^*} = 1$.

\begin{examples}\label{exec}\rm The following are scalar sequence spaces:\\
(i) $\lambda= \ell_p$ with $\lambda_*=\ell_{p^*}$ for every $1< p <\infty$.\\
(ii) $\lambda=c_0$ with $\lambda_*=\ell_1$.\\
(iii) $\lambda=\ell_1$ with $\lambda_*=\ell_\infty$.\\
 (iv) $\lambda=h_M$ with $\lambda_*=\ell_{M^*}$, where $M$ is an Orlicz function such that $M(1)=1$ and $M^*$ is the function complementary to $M$ (see \cite[Chapter 4]{lt1}).
\end{examples}

We omit the easy proof of the following result. The few facts that are not straightforward follow from condition (ii) of the definition and from the continuity of the biorthogonal functionals associated to the Schauder basis.

\begin{proposition} Let $\lambda$ be a scalar sequence space. \\
{\rm (a)} $\lambda_*$ is a linear space with the usual algebraic operations, the map
	$$\|\cdot \|_{\lambda_*}\colon \lambda_* \to [0,\infty)~,~ \|(\varphi(e_j))_{j=1}^\infty \|_{\lambda_*}=\|\varphi\|, $$
is a norm on $\lambda_*$ and the correspondence
$$\varphi \in \lambda^* \mapsto (\varphi(e_j))_{j=1}^\infty \in \lambda_*$$
is an isometric isomorphism. In particular, $\lambda_*$ is a Banach space.\\
{\rm (b)} For every $(\alpha_j)_{j=1}^\infty\in \lambda$, $(\alpha_j)_{j=1}^\infty=\sum\limits_{j=1}^\infty \alpha_j e_j$.\\
{\rm (c)} $c_{00} \subseteq \lambda_* \stackrel{1}{\hookrightarrow} \ell_\infty$.\\
{\rm (d)} $\|e_j\|_{\lambda_*}=1$ for every $j\in\mathbb{N}$.\label{C4S0P1}
\end{proposition}

Now we can define when a sequence class is represented by an operator ideal.

\begin{definition}\rm
A sequence class $X$ is said to be  \textit{ideal-representable} if there exists a Banach operator ideal $\cal I$ and a scalar sequence space $\lambda$  such that, for each Banach space $E$, the operator
	$$ \Psi\colon  \mathcal{I}(\lambda, E)\longrightarrow X(E)~, ~\Psi(u)=(u(e_j))_{j=1}^\infty,$$
is a well defined isometric isomorphism.  In this case, this fact shall be denoted by the symbol $\mathcal{I}(\lambda, \cdot)\approx X$, and we say that $X$ is $(\mathcal{I},\lambda)$-representable, or that $X$ is $\mathcal{I}$-representable or that the Banach operator ideal $\mathcal{I}$ represents $X$.
\end{definition}

\begin{remark}\rm (a) Although the operator $\Psi$ depends on the Banach space $E$, for simplicity we just write $\Psi$.\\
(b) If the operator $\Psi$  is an isometric isomorphism for a fixed Banach space $E$, this fact shall be denoted by $\mathcal{I}(\lambda, E)\approx X(E)$.\\
(c) If the operator $\Psi$ is well defined, then its linearity and injectivity are automatic.
\end{remark}

\begin{example}\rm As we saw in the Introduction, for $1 < p < \infty$, the class $\ell_p^w$ is $({\cal L}, \ell_{p^*})$-representable, $\ell_p^u$ is $({\cal K}, \ell_{p^*})$-representable and $\ell_p\langle \cdot \rangle$ is $({\cal N}, \ell_{p^*})$-representable. Moreover, $\ell_\infty(\cdot)$ is $({\cal L}, \ell_1)$-representable, $\ell_1^w$ is $({\cal L}, c_0)$-representable and $\ell_1^u$ is $({\cal K}, c_0)$-representable.
\end{example}

Next we prove some properties of ideal-representable classes which will be helpful later, in particular to give the first examples of non ideal-representable classes.

\begin{proposition}\label{p9me} If the sequence class $X$ is $(\mathcal{I},\lambda)$-representable, then $X$ is linearly stable and $X(\mathbb{K})\stackrel{1}{=}\lambda_* \stackrel{1}{=} \lambda^*$. In particular, if $X$ is ideal-representable, then $X(\mathbb{K})$ is a dual space.
\end{proposition}

\begin{proof} The last isometric equality was established in Proposition  \ref{C4S0P1}. Given $u\in \mathcal{L}(E,F)$ and $x=(x_j)_{j=1}^\infty \in  X(E)$, the surjectivity of $\Psi$ gives $u_x\in \mathcal{I}(\lambda,E)$ such that $(u_x(e_j))_{j=1}^\infty=(x_j)_{j=1}^\infty$. The ideal property of $\cal I$ gives  $u\circ u_x \in \mathcal{I}(\lambda,F)$, hence
	$$ (u(x_j))_{j=1}^\infty = (u\circ u_x(e_j))_{j=1}^\infty \in X(F)$$
because $X$ is $\mathcal{I}$-representable. This proves that  $\Pi_{X;X}(E,F)=\mathcal{L}(E,F)$. Moreover,
	\begin{align*}
		\|u\|_{X;X}&=\sup_{x=(x_j)_{j=1}^\infty \in B_{X(E)}} \|(u(x_j))_{j=1}^\infty\|_{X(F)}
		=\sup_{x \in B_{X(E)}} \|(u\circ u_x(e_j))_{j=1}^\infty\|_{X(F)}\\
		&=\sup_{x \in B_{X(E)}} \|u\circ u_x\|_{\mathcal{I}}\leq \|u\| \cdot \sup_{x \in B_{X(E)}}  \|u_x\|_{\mathcal{I}} = \|u\| \cdot \sup_{x \in B_{X(E)}}  \|\Psi(u_x)\|_{X(E)}\\
		&=\|u\|\cdot \sup_{x \in B_{X(E)}} \|(u_x(e_j))_{j=1}^\infty\|_{X(E)}
		=\|u\|\cdot \sup_{(x_j)_{j=1}^\infty \in B_{X(E)}} \|(x_j)_{j=1}^\infty\|_{X(E)}=\|u\|.
	\end{align*}
To show the reverse inequality, let $x \in E$ be given. 
Since $(x,0,0,\ldots) \in c_{00}(E)\subseteq X(E)$ and $X$ is $(\mathcal{I},\lambda)$-representable, there is  $u_x\in \mathcal{I}(\lambda,E)$ so that $u_x(e_1) = x$, $u_x(e_j) = 0$ for every $j > 1$ and 
 $\|(x,0,0,\ldots)\|_{X(E)}= \|u_x\|_{\mathcal{I}}$. Considering the functional $\varphi$ and the operator $\phi$ given by	$$\varphi\colon \lambda \longrightarrow \mathbb{K} ~,~ \varphi((\beta_j)_{j=1}^\infty):=\beta_1~~,~~\phi\colon \mathbb{K} \longrightarrow E~,~ \phi(\alpha):=\alpha x,  $$
it is easy to check that 
	$\|\varphi\|\leq 1$, $\|\phi\|=\|x\|$ and 
$u_x =  \phi\circ {\rm id}_{\mathbb{K}}\circ \varphi$.   Therefore, for every $x \in E$, using that $u$ is $(X;X)$-summing,
	\begin{align*}
		\|u(x)\|&=\|(u(x),0,0,\ldots)\|_\infty \leq \|(u(x),0,0,\ldots)\|_{X(F)} \leq \|u\|_{X;X}\cdot  \|(x,0,0,\ldots)\|_{X(E)}\\
&= \|u\|_{X;X}\cdot \|u_x\|_{\mathcal{I}} =\|u\|_{X;X}\cdot \|\phi\circ {\rm id}_{\mathbb{K}}\circ \varphi\|_{\mathcal{I}} \leq \|u\|_{X;X}\cdot \|x\|.
	\end{align*}
This proves that $\|u\|\leq \|u\|_{X;X}$ and establishes the linear stability of  $X$. 	

Finally, we have ${\cal I}(\lambda, \mathbb{K}) \stackrel{1}{=} {\cal L}(\lambda, \mathbb{K}) = \lambda^*$ because $\cal I$ is a Banach operator ideal. So,
	\begin{align*}
		X(\mathbb{K})=\Psi\left( \mathcal{I}(\lambda,\mathbb{K}) \right)
		&=\left\{ (\varphi(e_j))_{j=1}^\infty:\varphi\in \mathcal{I}(\lambda,\mathbb{K}) \right\}=
		\left\{ (\varphi(e_j))_{j=1}^\infty:\varphi\in \lambda^* \right\}=\lambda_* \mbox{~\,and}
	\end{align*}
$$\|(\varphi(e_j))_{j=1}^\infty\|_{X(\mathbb{K})}=\|\varphi\|_{\mathcal{I}}=\|\varphi\|=\|(\varphi(e_j))_{j=1}^\infty\|_{\lambda_*}.$$
\end{proof}

  The proposition above says that non linearly stable classes are not ideal-representable. Next we give examples of non ideal-representable linearly stable sequence classes.

 \begin{example}\rm It is well known that $c_0$ is not isomorphic to a dual space, so the linearly stable sequence classes $c_0(\cdot)$ and $c_0^w$ are not ideal-representable by Proposition \ref{p9me}. The space $c$ of convergent sequences is separable and contains a copy of $c_0$, hence $c$ is not isomorphic to a dual space by \cite[Theorem 6.13]{carothers}. Therefore the linearly stable sequence class $c(\cdot)$ is not ideal-representable either. Further counterexamples shall be given later.
 \end{example}

 Our next purposes are to characterize ideal-representable classes and to exhibit an operator ideal that represents a given ideal-representable class. The next H\"older-type inequality follows easily from Proposition \ref{C4S0P1}(b).

 \begin{lemma} \label{C4S1L2} If $\lambda$ is a scalar sequence space, $(\alpha_j)_{j=1}^\infty \in \lambda$ and $(\beta_j)_{j=1}^\infty \in \lambda_*$, then the series $\sum\limits_{j=1}^\infty \alpha_j \beta_j$ converges and $$\left| \sum_{j=1}^\infty \alpha_j \beta_j \right|\leq \|(\alpha_j)_{j=1}^\infty\|_{\lambda} \cdot \|(\beta_j)_{j=1}^\infty\|_{\lambda_*}.$$
	In particular, $\left| \sum\limits_{j=n}^m \alpha_j \beta_j \right|\leq \left\|\sum\limits_{j=n}^{m} \alpha_j e_j\right\|_{\lambda} \cdot \|(\beta_j)_{j=1}^\infty\|_{\lambda_*}$ for all $m>n\in\mathbb{N}$.
 \end{lemma}

    A sequence class $X$ is {\it finitely determined} if, for any Banach space $E$ and every $E$-valued sequence $(x_j)_{j=1}^\infty$, it holds $(x_j)_{j=1}^\infty \in X(E)$ if and only if $\sup\limits_{k \in \mathbb{N}}\|(x_j)_{j=1}^k \|_{X(E)}< +\infty$, and, in this case, $\|(x_j)_{j=1}^\infty \|_{X(E)}= \sup\limits_{k \in \mathbb{N}}\|(x_j)_{j=1}^k \|_{X(E)} $ (see \cite{botelho+campos}). Here, $(x_j)_{j=1}^k = (x_1, \ldots, x_k, 0,0,\ldots)$. The classes $\ell_p(\cdot), \ell_p^w, \ell_\infty$, RAD, $\ell_p\langle\cdot\rangle$ and $\ell_p^{\rm mid}$ are finitely determined, whereas the classes $c_0(\cdot), c_0^w$, $c(\cdot), \ell_p^{u}$ and Rad are not. 

\begin{proposition} \label{C4S1P2}
Let $\lambda$ be a scalar sequence space. Defining, for each Banach space  $E$,
	$$\lambda_*^w(E):=\left\{(x_j)_{j=1}^\infty  \in E^{\mathbb{N}} : (\varphi(x_j))_{j=1}^\infty \in \lambda_*  {\rm ~for~every~}\varphi\in E^* \right\},$$
we have:

 \noindent {\rm (i)} $ \|(x_j)_{j=1}^\infty\|_{w,\lambda_*}:=\sup\limits_{\varphi \in B_{E^*}} \|(\varphi(x_j))_{j=1}^\infty\|_{\lambda_*}<+\infty $ for every $(x_j)_{j=1}^\infty\in \lambda_*^w(E)$.\\
 \noindent {\rm (ii)} $\lambda_*^w(E)$ is a linear subspace of $E^{\mathbb{N}}$ on which the map $\|\cdot\|_{w,\lambda_*}\colon \lambda_*^w(E) \longrightarrow [0,\infty)$ is a norm.\\
\noindent {\rm (iii)} The operator
		$$ \Psi\colon  \mathcal{L}(\lambda, E)\longrightarrow \lambda_*^w(E)~, ~\Psi(u)=(u(e_j))_{j=1}^\infty,$$
is a well defined isometric isomorphism. In particular,  $\left( \lambda_*^w(E),\|\cdot\|_{w,\lambda_*}\right)$ is a Banach space.\\
 \noindent {\rm (iv)} The rule $E \mapsto \lambda_*^w(E)$ is a   $(\mathcal{L},\lambda)$-representable sequence class with $\lambda_*^w(\mathbb{K})\stackrel{1}{=}\lambda_*$. In particular, the class $\lambda_*^w$ is linearly stable.\\
\noindent {\rm (v)} If $X$ is a linearly stable sequence class so that $X(\mathbb{K})\stackrel{1}{=}\lambda_*$, then $X\stackrel{1}{\hookrightarrow} \lambda_*^w$.\\
\noindent {\rm (vi)}	Suppose that $\lambda_*$ enjoys the following property: For each  $(\alpha_j)_{j=1}^\infty \in \mathbb{K}^{\mathbb{N}}$, $(\alpha_j)_{j=1}^\infty \in \lambda_*$ if and only if $\sup\limits_{n\in \mathbb{N}} \|(\alpha_j)_{j=1}^n\|_{\lambda_*}<+\infty$ and, in this case, $\|(\alpha_j)_{j=1}^\infty\|_{\lambda_*}=\sup\limits_{n\in \mathbb{N}} \|(\alpha_j)_{j=1}^n\|_{\lambda_*}$. Then the class $\lambda_*^w$ is finitely determined.
\end{proposition}

\begin{proof} Is is clear that $\lambda_*^w(E)$ is a linear subspace of $E^{\mathbb{N}}$. The definition of $\lambda_*^w(E)$ yields that $u\colon E^* \longrightarrow \lambda_*$ given by $u(\varphi)=(\varphi(x_j))_{j=1}^\infty$ is a well defined linear operator. Using that $\lambda_* \stackrel{1}{\hookrightarrow} \ell_\infty$ (Proposition \ref{C4S0P1}(c)), a standard closed graph argument gives the continuity of $u$, hence
$$\sup_{\varphi\in B_{E^*}} \|(\varphi(x_j))_{j=1}^\infty\|_{\lambda_*}=\sup_{\varphi\in B_{E^*}} \|u(\varphi)\|_{\lambda_*}=\|u\|<+\infty,$$
proving (i). The norm axioms are all easy (the implication $\|x\|_{w, \lambda_*} = 0 \Rightarrow x = 0$ depends on the Hahn-Banach Theorem), so (ii) is proved.

(iii) Given $u\in \mathcal{L}(\lambda,E)$ and $\varphi \in E^*$,  we have $\varphi \circ u \in \mathcal{L}(\lambda, \mathbb{K})=\lambda^*$. From the definition of  $\lambda_*$, $ (\varphi  (u(e_j)))_{j=1}^\infty= (\varphi \circ u(e_j))_{j=1}^\infty \in\lambda_*,$ thus $(u(e_j))_{j=1}^\infty\in \lambda_*^w(E)$. This proves that $\Psi$ is well defined. Its linearity is obvious. From
\begin{align*}		\|\Psi(u)\|_{w,\lambda_*}&=\|(u(e_j))_{j=1}^\infty\|_{w,\lambda_*}= 
\sup_{\varphi \in B_{E^*}} \|(\varphi \circ u(e_j))_{j=1}^\infty\|_{\lambda_*}    		=\sup_{\varphi \in B_{E^*}} \|\varphi \circ u\|\leq 
		\|u\|,
	\end{align*}
we obtain the continuity of $\Psi$. Combining this continuity with the fact that $(e_j)_{j=1}^\infty$ is a  Schauder basis for $\lambda$, we have $u=0 \Leftrightarrow u(e_j)=0$ for every $j\in\mathbb{N}$. This implies the injectivity of $\Psi$. Given $(x_j)_{j=1}^\infty \in \lambda_*^w(E)$, consider $u\colon \lambda \longrightarrow E$ given by $u((\alpha_j)_{j=1}^\infty)= \sum\limits_{j=1}^{\infty} \alpha_j x_j.$
 Given $(\alpha_j)_{j=1}^\infty \in \lambda$, the series $\sum\limits_{j=1}^{\infty} \alpha_j e_j$ converges in $\lambda$ because $(e_j)_{j=1}^\infty$ is a Schauder basis for $\lambda$. So, given $\varepsilon >0$ there exists $n_0\in\mathbb{N}$ so that  $$\left\|\sum_{j=n}^{m} \alpha_j e_j\right\|_{\lambda}< \frac{\varepsilon}{\|(x_j)_{j=1}^\infty\|_{w,\lambda_*}} \mbox{ whenever } m>n\geq n_0.$$
	Applying \ref{C4S1L2} we get
	\begin{align*}
		\left\|\sum_{j=n}^{m} \alpha_j x_j\right\|&=\sup_{\varphi \in B_{E^*} } \left| \sum_{j=n}^m \alpha_j \varphi(x_j)\right|\leq \sup_{\varphi \in B_{E^*} } \left\|\sum_{j=n}^{m} \alpha_j e_j\right\|_{\lambda}\cdot \|(\varphi(x_j))_{j=1}^\infty\|_{\lambda_*}\\
		&=  \left\|\sum_{j=n}^{m} \alpha_j e_j\right\|_{\lambda}\cdot\sup_{\varphi \in B_{E^*} }  \|(\varphi(x_j))_{j=1}^\infty\|_{\lambda_*}= \left\|\sum_{j=n}^{m} \alpha_j e_j\right\|_{\lambda}\cdot   \|(x_j)_{j=1}^\infty\|_{w,\lambda_*}<\varepsilon
	\end{align*}
  for all $m>n\geq n_0$. This proves that the series $\sum\limits_{j=1}^{\infty} \alpha_j x_j$ converges in $E$, which means that $u$ is well defined. Its linearity and the equality $(u(e_j))_{j=1}^\infty = (x_j)_{j=1}^\infty$ are obvious. By Lemma \ref{C4S1L2},
	\begin{align*}
		\left\|u((\alpha_j)_{j=1}^\infty)\right\|&=\left\|\sum_{j=1}^{\infty} \alpha_j x_j\right\|= \sup_{\varphi \in B_{E^*} } \left| \sum_{j=1}^\infty \alpha_j \varphi(x_j)\right|\leq \sup_{\varphi \in B_{E^*} } \|(\alpha_j)_{j=1}^\infty\|_{\lambda}\cdot \|(\varphi(x_j))_{j=1}^\infty\|_{\lambda_*}\\
		&=\|(\alpha_j)_{j=1}^\infty\|_{\lambda}\cdot \|(x_j)_{j=1}^\infty\|_{w,\lambda_*}=\|(\alpha_j)_{j=1}^\infty\|_{\lambda}\cdot \|(u(e_j))_{j=1}^\infty\|_{w,\lambda_*}.
	\end{align*}
This gives that $u$ is continuous and $\|u\|\leq \|(u(e_j))_{j=1}^\infty\|_{w,\lambda_*} = \|\Psi(u)\|_{w,\lambda_*}$, which proves (iii).

(iv) Given a Banach space $E$, we already know that $\lambda_*^w(E)$ is a Banach space of $E$-valued sequences. From Proposition \ref{C4S0P1}(c) and the Hahn-Banach Theorem it follows easily that $c_{00}(E)\subseteq \lambda_*^w(E) \stackrel{1}{\hookrightarrow} \ell_\infty(E)$. Combining the isomorphism from item (iii) for $E = \mathbb{K}$ with the one from Proposition \ref{C4S0P1}, we obtain $\lambda_*^w(\mathbb{K})\stackrel{1}{=} \lambda_*$. Thus, $\|e_j\|_{w,\lambda_*}=\|e_j\|_{\lambda_*}=1$ for every $j\in\mathbb{N}$, where the last equality comes from Proposition \ref{C4S0P1}(d). It is proved that the rule  
$E \mapsto \lambda_*^w(E)$ is a sequence class, which is $(\mathcal{L},\lambda)$-representable by item (iii). The linear stability of $ \lambda_*^w$ follows from Proposition \ref{p9me}.

(v) The statement follows immediately from the linear stability of $X$ applied to the functionals $\varphi \in E^*$.

(vi) Let $(x_j)_{j=1}^\infty \in E^\mathbb{N}$ be given. Suppose that $(x_j)_{j=1}^\infty \in  \lambda_*^w(E)$. For each $n \in \mathbb{N}$, we know that $(x_j)_{j=1}^n \in c_{00}(E)\subseteq \lambda_*^w(E)$ and that, for each $\varphi \in E^*$, the linear stability of $\lambda_*^w$ gives $(\varphi(x_j))_{j=1}^\infty \in \lambda_*^w(\mathbb{K}) \stackrel{1}{=} \lambda_*$. Applying the property that $\lambda_*$ enjoys by assumption we get
\begin{align*}
		\|(x_j)_{j=1}^n\|_{w,\lambda_*}&=\sup_{\varphi \in B_{E^*}}\|(\varphi(x_j))_{j=1}^n\|_{\lambda_*}\leq \sup_{\varphi \in B_{E^*}}\|(\varphi(x_j))_{j=1}^\infty\|_{\lambda_*}=\|(x_j)_{j=1}^\infty\|_{w,\lambda_*}.
	\end{align*}
Thus, $\sup\limits_{n\in\mathbb{N}} \|(x_j)_{j=1}^n\|_{w,\lambda_*} \leq \|(x_j)_{j=1}^\infty\|_{w,\lambda_*} < +\infty$. Conversely, suppose that $\sup\limits_{n\in\mathbb{N}} \|(x_j)_{j=1}^n\|_{w,\lambda_*} < +\infty$.   For $n\in\mathbb{N}$ and  $0 \neq \phi \in E^*$,
	\begin{align}    		\|(\phi(x_j))_{j=1}^n\|_{\lambda_*}&\!=\!\|\phi\|\cdot\left\|\left(\frac{\phi}{\|\phi\|}(x_j)\right)_{j=1}^n\right\|_{\lambda_*}
		\!\leq \|\phi\|\cdot \sup_{\varphi \in B_{E^*}} \|(\varphi(x_j))_{j=1}^n\|_{\lambda_*}
		\!=\!\|\phi\|\!\cdot\!  \|(x_j)_{j=1}^n\|_{w,\lambda_*}.\label{3fyc}
	\end{align}
It follows that  $\sup\limits_{n\in\mathbb{N}}	\|(\phi(x_j))_{j=1}^n\|_{\lambda_*}< +\infty $. Using again the property enjoyed by $\lambda_*$ by assumption, we get  $(\phi(x_j))_{j=1}^\infty \in \lambda_*$, that is, $(x_j)_{j=1}^\infty \in \lambda_*^w(E)$. The same property also gives
	\begin{align*}
		\|(x_j)_{j=1}^\infty\|_{w,\lambda_*}&=\sup_{\varphi \in B_{E^*}}\|(\varphi(x_j))_{j=1}^\infty\|_{\lambda_*}= \sup_{\varphi \in B_{E^*}} \sup_{n\in\mathbb{N}}\|(\varphi(x_j))_{j=1}^n\|_{\lambda_*}\\
		&\stackrel{\small(\ref{3fyc})}{\leq} \sup_{\varphi \in B_{E^*}} \sup_{n\in\mathbb{N}} \|\varphi\|\cdot  \|(x_j)_{j=1}^n\|_{w,\lambda_*}=  \sup_{n\in\mathbb{N}}   \|(x_j)_{j=1}^n\|_{w,\lambda_*},
	\end{align*}
proving that the sequence class $\lambda_*^w$ is finitely determined.
\end{proof}

We get from the proposition above a large family of ideal-representable sequence classes:

\begin{example}\label{lm9w}\rm Let $M$ be an Orlicz function with $M(1)= 1$. Taking $\lambda = h_M$, from Example \ref{exec} and Proposition \ref{C4S1P2} we get that the rule
$$E \mapsto \ell_{M^*}^w(E):=\left\{(x_j)_{j=1}^\infty  \in E^{\mathbb{N}} : (\varphi(x_j))_{j=1}^\infty \in \ell_{M^*}  {\rm ~for~every~}\varphi\in E^* \right\}$$
is a $(\mathcal{L}, h_M)$-representable sequence class, that is, $ \mathcal{L}(h_M,\cdot) \approx \ell_{M^*}^w$. In particular, this sequence class is linearly stable.
\end{example}

 Next we introduce the operator ideal that is universal for the representation of sequence classes, meaning that if a class is ideal-representable, then it is represented by this ideal.

 \begin{proposition} \label{C4S1P3}
	Let $\lambda$ be a scalar sequence space and let $X$ be a linearly stable sequence class with $X(\mathbb{K})\stackrel{1}{=}\lambda_*$. Then $\Pi_{\lambda_*^w;X}$ is a Banach operator ideal and, for every Banach space $F$, the operator
	$$ \Psi \colon \Pi_{\lambda_*^w;X}(\lambda,F) \longrightarrow X(F)~,~\Psi(u)=(u(e_j))_{j=1}^\infty,$$
is well defined, linear, injective and continuous with $\|\Psi\|\leq 1$.
 \end{proposition}

 \begin{proof} The class $X$ is linearly stable by assumption and the $\lambda_*^w$ is linearly stable with $\lambda_*^w(\mathbb{K})\stackrel{1}{=} \lambda_*\stackrel{1}{=}X(\mathbb{K})$ by Proposition \ref{C4S1P2}(iv). Therefore, $\Pi_{\lambda_*^w;X}$ is a  Banach operator ideal (see Section 2).  Let $u\in \Pi_{\lambda_*^w;X}(\lambda,F)$ be given. For any $\varphi \in \lambda^*$, $(\varphi(e_j))_{j=1}^\infty \in \lambda_*$ by the definition of $\lambda_*$, hence $(e_j)_{j=1}^\infty \in \lambda_*^w(\lambda)$ by the definition of $\lambda_*^w(\lambda)$. We have $(u(e_j))_{j=1}^\infty \in X(F)$ because $u$ is $(\lambda_*^w,X)$-summing, so $\Psi$ is well defined. The linearity and injectivity are clear. Moreover, using the definition of the norms $\|\cdot\|_{w,\lambda_*}$ and $\|\cdot\|_{\lambda_*}, $
	\begin{align*}
		\|\Psi(u)\|_{X(F)}&=\|(u(e_j))_{j=1}^\infty\|_{X(F)}\leq \|u\|_{\lambda_*^w; X} \cdot \|(e_j)_{j=1}^\infty\|_{w,\lambda_*} \\ &=  \|u\|_{\lambda_*^w; X} \cdot \sup_{\varphi \in B_{\lambda^*}} \|(\varphi(e_j))_{j=1}^\infty\|_{\lambda_*}
		= \|u\|_{\lambda_*^w; X}\cdot \sup_{\varphi \in B_{\lambda^*}} \|\varphi\| = \|u\|_{\lambda_*^w; X},
	\end{align*}
	which proves that $\|\Psi\|\leq 1$.	
\end{proof}

Recall that $\cal N$ is the ideal of nuclear operators. The proofs of the representation $ \mathcal{N}(c_0,\cdot) \approx \ell_1(\cdot)$ we are aware of depend on topological tensor products techniques. As a first application of Proposition \ref{C4S1P3}, we give a simple tensor product-free proof of this fact. It is worth remarking that $\ell_1(\cdot) \stackrel{1}{=} \ell_1 \langle \cdot \rangle$ \cite[Corollary 3.9]{fourie+rontgen}. As usual, for $1 \leq p < \infty$, the ideal of absolutely $p$-summing operators is denoted by $\Pi_p$, that is, $\Pi_p = \Pi_{\ell_p^w; \ell_p(\cdot)}$.

\begin{proposition} \label{i9vq} The class $\ell_1(\cdot)$ is $\cal N$-representable.
\end{proposition}

\begin{proof} Given a Banach space $F$, by Proposition \ref{C4S1P3} we just have to show that the operator
	$ \Psi \colon \Pi_1(c_0,F) \longrightarrow \ell_1(F)$, $\Psi(u)=(u(e_j))_{j=1}^\infty$, is a surjective isometry and that ${\cal N}(c_0,F) \stackrel{1}{=}\Pi_1(c_0,F)$. For the latter isometric isomorphism, see \cite[Ex.\,11.2, p.\,142]{df}. To see the surjectivity, given $(x_j)_{j=1}^\infty \in \ell_1(F)$, consider the biorthogonal functionals $(e_j^*)_{j=1}^\infty$ associated to the Schauder basis $(e_j)_{j=1}^\infty$ of $c_0$. From $\sum\limits_{j=1}^\infty \|e_j^*\|\cdot\|x_j\| = \sum\limits_{j=1}^\infty \|x_j\| < \infty$ it follows that $u \colon c_0 \longrightarrow F$ given by $u\left(y \right) = \sum\limits_{n=1}^\infty e_n^*(y)x_n,$
is a nuclear operator. Since  $u(e_j) = \sum\limits_{n=1}^\infty e_n^*(e_j)x_n = x_j$ for every $j$, we have $\Psi(u) = (u(e_j))_{j=1}^\infty = (x_j)_{j=1}^\infty $. The inequality $\|u\|_{\cal N} \leq \|(x_j)_{j=1}^\infty\|_{1}$ is obvious and the reverse inequality follows from Proposition \ref{C4S1P3}.
\end{proof}

  The proof above shows that the class $\ell_1(\cdot)$ is ${\cal N}$-representable and $\Pi_1$-representable. Since these two ideals are different, this is an example of the non uniqueness of the ideal that represents an ideal-representable class.

In the next result we prove two useful characterizations of ideal-representable classes and we exhibit a concrete ideal which represents a representable class. From Proposition \ref{p9me} we know that the study of ideal-representable classes can be restricted to linearly stable classes $X$ for which there is a scalar sequence space $\lambda$ so that $X(\mathbb{K})\stackrel{1}{=}\lambda_*$.

\begin{theorem} \label{C4S1T1} Let $X$ be a linearly stable sequence class and let $\lambda$ be a scalar sequence space such that $X(\mathbb{K})\stackrel{1}{=}\lambda_*$. The following are equivalent:\\
\noindent{\rm (i)} $X$ is ideal-representable.\\
\noindent{\rm (ii)}	
If $E$ is a Banach space, $v\in \mathcal{L}(\lambda,\lambda)$ and $u\in \mathcal{L}(\lambda,E)$ is so that $(u(e_j))_{j=1}^\infty\in X(E)$, then $$(u\circ v(e_j))_{j=1}^\infty\in X(E) \mbox{~~and~~}\|(u\circ v(e_j))_{j=1}^\infty\|_{X(E)}\leq \|v\|\cdot \|(u(e_j))_{j=1}^\infty\|_{X(E)}.$$
\noindent{\rm (iii)} 
 If $E$ is a Banach space, $(x_j)_{j=1}^\infty \in X(E)$ and  $((\alpha_{j,i})_{i=1}^\infty)_{j=1}^\infty \in \lambda_*^w(\lambda)$, then  		$ \left( \sum\limits_{i=1}^\infty \alpha_{j,i} x_i\right)_{j=1}^\infty  \in X(E)$ and
		$$\left\| \left( \sum_{i=1}^\infty \alpha_{j,i} x_i\right)_{j=1}^\infty\right\|_{X(E)} \leq \left\| ((\alpha_{j,i})_{i=1}^\infty)_{j=1}^\infty\right\|_{w,\lambda_*} \cdot \|(x_j)_{j=1}^\infty\|_{X(E)}.$$
\noindent{\rm (iv)}  $X$ is $\Pi_{\lambda_*^w;X}$-representable.
\end{theorem}


\begin{proof} Let $E$ be an arbitrary Banach space.

	(i)$\Rightarrow$ (ii) Suppose that $X$ is  $(\mathcal{I},\lambda)$-representable. 
Let $v\in \mathcal{L}(\lambda,\lambda)$ and  $u\in\mathcal{L}(\lambda,E)$ with $(u(e_j))_{j=1}^\infty \in X(E)$ be given. Since $\mathcal{I}(\lambda,\cdot)\approx X$ and $(u(e_j))_{j=1}^\infty \in X(E)$, we have $u\in \mathcal{I}(\lambda,E)$. From the ideal property we get $ u\circ v\in \mathcal{I}(\lambda,E)$.
	Using again the representation  $\mathcal{I}(\lambda,\cdot)\approx X$, it follows that $(u\circ v(e_j))_{j=1}^\infty \in X(E)$ and
	\begin{align*}
		\|(u\circ v(e_j))_{j=1}^\infty\|_{X(E)}&=\|u\circ v\|_{\mathcal{I}}
		\leq \|v\|\cdot \|u\|_{\mathcal{I}}
		=  \|v\|\cdot  \|(u(e_j))_{j=1}^\infty\|_{X(E)}.
	\end{align*}
	
	(ii) $\Rightarrow$ (iii) and (ii) $\Rightarrow$ (iv) 
Let $(x_j)_{j=1}^\infty \in X(E)$ and $((\alpha_{j,i})_{i=1}^\infty)_{j=1}^\infty \in \lambda_*^w(\lambda)$ be given. Since $X$ is linearly stable and  $X(\mathbb{K})\stackrel{1}{=}\lambda_*$, Proposition \ref{C4S1P2}(v) gives $X(E)\stackrel{1}{\hookrightarrow} \lambda_*^w(E)$. Applying item (iii) of the same proposition for $E$ and for $\lambda$, there are $u\in \mathcal{L}(\lambda,E)$ and $v\in \mathcal{L}(\lambda,\lambda)$ such that $(u(e_j))_{j=1}^\infty=(x_j)_{j=1}^\infty $ and  $(v(e_j))_{j=1}^\infty=((\alpha_{j,i})_{i=1}^\infty)_{j=1}^\infty$. Thus,
	\begin{align*}
		(u\circ v(e_j))_{j=1}^\infty&=\left(  u\left(v(e_j)\right)  \right)_{j=1}^\infty=\left(  u\left((\alpha_{j,i})_{i=1}^\infty\right)  \right)_{j=1}^\infty=  \left(u\left( \sum_{i=1}^\infty \alpha_{j,i} e_i\right)\right)_{j=1}^\infty \\& =  \left( \sum_{i=1}^\infty \alpha_{j,i} u(e_i)\right)_{j=1}^\infty  		
		=\left( \sum_{i=1}^\infty \alpha_{j,i} x_i\right)_{j=1}^\infty.
	\end{align*}
By assumption, (ii) holds, so 
	$\left( \sum\limits_{i=1}^\infty \alpha_{j,i} x_i\right)_{j=1}^\infty=(u\circ v(e_j))_{j=1}^\infty \in X(E)$ and
		\begin{align*}
		\left\| \left( \sum_{i=1}^\infty \alpha_{j,i} x_i\right)_{j=1}^\infty\right\|_{X(E)}&=\|(u\circ v(e_j))_{j=1}^\infty \|_{X(E)}\leq  \|v\|\cdot \|(u(e_j))_{j=1}^\infty\|_{X(E)}\\&= \left\| ((\alpha_{j,i})_{i=1}^\infty)_{j=1}^\infty\right\|_{w,\lambda_*} \cdot \|(x_j)_{j=1}^\infty\|_{X(E)}.
	\end{align*}
	This proves (iii).	By Proposition \ref{C4S1P3} we know that $\Pi_{\lambda_*^w;X}$ is a Banach operator ideal and that the operator
	$ \Psi\colon \Pi_{\lambda_*^w;X}( \lambda, E) \longrightarrow X(E)$ 
is linear injective with
$ \|\Psi(u)\|=\|(u(e_j))_{j=1}^\infty\|_{X(E)}\leq \|u\|_{\lambda_*^w;X}$.
	We also have $(x_j)_{j=1}^\infty \in X(E) \stackrel{1}{\hookrightarrow} \lambda_*^w(E)$ and $(u(e_j))_{j=1}^\infty=(x_j)_{j=1}^\infty$ for some $u \in \mathcal{L}(\lambda,E)$. To show that $u\in \Pi_{\lambda_*^w;X}(\lambda,E)$, let   $(\alpha_j)_{j=1}^\infty \in \lambda_*^w(\lambda)$ be given. Calling on Proposition \ref{C4S1P2} once again, there is $w \in \mathcal{L}(\lambda,\lambda)$ such that $(w(e_j))_{j=1}^\infty=(\alpha_j)_{j=1}^\infty$. And applying again assumption (ii), we get 
	$(u(\alpha_j))_{j=1}^\infty=(u\circ w(e_j))_{j=1}^\infty \in X(E), $
	which implies $u\in \Pi_{\lambda_*^w;X}(\lambda,E)$, and
	\begin{align*}
		\|(u(\alpha_j))_{j=1}^\infty\|_{X(E)}&= \|(u\circ w(e_j))_{j=1}^\infty\|_{X(E)}\leq \|w\|\! \cdot \! \|(u(e_j))_{j=1}^\infty\|_{X(E)}= \|(\alpha_j)_{j=1}^\infty\|_{w,\lambda_*}\!\cdot\! \|(x_j)_{j=1}^\infty\|_{X(E)}.
	\end{align*}
Taking the supremum over $(\alpha_j)_{j=1}^\infty \in B_{\lambda_*^w(\lambda)}$, it follows that $$\|u\|_{\lambda_*^w;X}\leq \|(x_j)_{j=1}^\infty\|_{X(E)}=\|(u(e_j))_{j=1}^\infty\|_{X(E)}.$$
It is proved that $\Psi$ is an isometric isomorphism, regardless of the Banach space $E$. This proves that 
$X$ is $\Pi_{\lambda_*^w;X}$-representable, which establishes (iv).
	
The implication	(iv)$\Rightarrow$ (i) is immediate, so it is enough to prove
	(iii) $\Rightarrow$ (ii) Assume (iii) and let $v\in \mathcal{L}(\lambda,\lambda)$ and $u\in \mathcal{L}(\lambda,E)$ with $(u(e_j))_{j=1}^\infty\in X(E)$ be given. By Proposition \ref{C4S1P2}(iii) we have  $(v(e_j))_{j=1}^\infty\in \lambda_*^w(\lambda) $ with $\| (v(e_j))_{j=1}^\infty\|_{w,\lambda_*}=\|v\|$. Writing  $(v(e_j))_{j=1}^\infty=((\alpha_{j,i})_{i=1}^\infty)_{j=1}^\infty$, by assumption we have
	\begin{align*}
		(u\circ v(e_j))_{j=1}^\infty&= \left(  u((\alpha_{j,i})_{i=1}^\infty)  \right)_{j=1}^\infty
		=\left( u\left(\sum_{i=1}^\infty \alpha_{j,i}e_i\right)  \right)_{j=1}^\infty
		=\left( \sum_{i=1}^\infty \alpha_{j,i} u(e_i)\right)_{j=1}^\infty \in X(E)
	\end{align*}
	and
	\begin{align*}
		\|(u\circ v(e_j))_{j=1}^\infty\|_{X(E)}&=\left\| \left( \sum_{i=1}^\infty \alpha_{j,i} u(e_i)\right)_{j=1}^\infty  \right\|_{X(E)}
		\leq \left\| ((\alpha_{j,i})_{i=1}^\infty)_{j=1}^\infty\right\|_{w,\lambda_*} \cdot \|(u(e_j))_{j=1}^\infty\|_{X(E)}\\
		&= \left\| (v(e_j))_{j=1}^\infty\right\|_{w,\lambda_*} \cdot \|(u(e_j))_{j=1}^\infty\|_{X(E)}
		= \| v\|\cdot \|(u(e_j))_{j=1}^\infty\|_{X(E)}.
	\end{align*} It is proved that (ii) holds.
\end{proof}

\begin{example}\label{C4S1E3}\rm
	Let  $1<p<\infty$ and let $M$ be an Orlicz function with $M(1)=1$. The following representations follow from the examples in the Introduction, Example \ref{lm9w} and Theorem \ref{C4S1T1}:\\
\noindent$\bullet$ $\Pi_{\ell_{M^*}^w;\ell_{M^*}^w}(h_{M}, \cdot) \approx\ell_{M^*}^w$, $\Pi_{\ell_p^w;\ell_p^w}(\ell_{p^*}, \cdot) \approx\ell_p^w$ and $\Pi_{\ell_1^w;\ell_1^w}(c_0, \cdot) \approx\ell_1^w$. These representations can be regarded as trivial because, due to the linear stability of the classes $\ell_{M^*}^w$, $\ell_p^w$ and $\ell_1^w$, it holds $\Pi_{\ell_{M^*}^w;\ell_{M^*}^w}=\Pi_{\ell_p^w;\ell_p^w}=\Pi_{\ell_1^w;\ell_1^w}=\mathcal{L}$. \\
\noindent$\bullet$ $\Pi_{\ell_p^w;\ell_p^u}(\ell_{p^*}, \cdot)  \approx\ell_p^u$, hence $\mathcal{K}(\ell_{p^*},\cdot)\stackrel{1}{=}\Pi_{\ell_p^w;\ell_p^u}(\ell_{p^*}, \cdot)$.\\
\noindent$\bullet$ $\Pi_{\ell_1^w;\ell_1^u}(c_0, \cdot) \approx\ell_1^u$, hence  $\mathcal{K}(c_0,\cdot)\stackrel{1}{=}\Pi_{\ell_1^w;\ell_1^u}(c_0, \cdot)$.\\
\noindent$\bullet$ $\Pi_{\ell_p^w;\ell_p\langle \cdot\rangle}(\ell_{p^*}, \cdot)  \approx\ell_p\langle \cdot\rangle$, hence  $\mathcal{N}(\ell_{p^*},\cdot)\stackrel{1}{=}\Pi_{\ell_p^w;\ell_p\langle \cdot\rangle}(\ell_{p^*}, \cdot)$.
\end{example}

The representations above are either known or expected. Our next purpose is to give new/unexpected consequences of Theorem \ref{C4S1T1}. 
Note that, since Rad, RAD and $\ell_p(\cdot)$, $1 < p < \infty$, are linearly stable and ${\rm Rad}(\mathbb{K}) = {\rm RAD}(\mathbb{K})= \ell_2 = \ell_2^*$ and $\ell_p(\mathbb{K}) = \ell_p = \ell_{p^*}^*$, we cannot use Proposition \ref{p9me}  to conclude that these classes are not ideal-representable. We will reach this conclusion by using Theorem \ref{C4S1T1}. 

\begin{corollary}\label{C4S1C1} The sequence class {\rm Rad} is not ideal representable.
\end{corollary}

\begin{proof} Let $E$ be a Banach space with no finite cotype, for instance, $E = c_0$ (for the theory of cotype of Banach spaces, see \cite[Chapter 11]{diestel+jarchow+tonge}). Applying \cite[Proposition 2]{ricardo+tien} for the orthonormal basis $(e_j)_{j=1}^\infty$ of the Hilbert space $\ell_2$, 
there is an operator $u \in {\cal L}(\ell_2, E)$ for which the following equivalence does not hold:
	\begin{equation*}\label{pm9q} (u(e_j))_{j=1}^\infty \in {\rm Rad}(E)   \Longleftrightarrow u \in \Pi_{\ell_2^w;{\rm Rad}}(\ell_2,E).  \end{equation*}
Since $(e_j)_{j=1}^\infty \in \ell_2^w(\ell_2)$, the implication $(\Longleftarrow)$ above holds for any operator in $ {\cal L}(\ell_2, E)$, so the implication $(\Longrightarrow)$ is the one that fails for $u$. Therefore,  $(u(e_j))_{j=1}^\infty \in {\rm Rad}(E)$ and $u \notin \Pi_{\ell_2^w;{\rm Rad}}(\ell_2,E)$. Assume that Rad is ideal representable. As ${\rm Rad}(\mathbb{K}) = \ell_2 = \ell_2^*$ and Rad is linearly stable, the ideal $\Pi_{\ell_2^w;\text{Rad}}$ represents Rad by Theorem \ref{C4S1T1}, hence the operator $\Psi \colon \Pi_{\ell_2^w;\text{Rad}}(\ell_{2},E) \longrightarrow \text{Rad}(E)$, $\Psi(v) = (v(e_j))_{j=1}^\infty$, is surjective.
Since $(u(e_j))_{j=1}^\infty \in {\rm Rad}(E)$, there is $v \in \Pi_{\ell_2^w;{\rm Rad}}(\ell_2,E)$ such that 
	$$(u(e_j))_{j=1}^\infty = \Psi(v) = (v(e_j))_{j=1}^\infty. $$
So, $u(e_j) = v(e_j)$ for every $j$, from which it follows that $u = v  \in \Pi_{\ell_2^w;{\rm Rad}}(\ell_2,E)$ because $(e_j)_{j=1}^\infty$ is a Schauder basis for $\ell_2$. This contradiction proves that Rad is not ideal-representable. 
\end{proof}

We shall settle the case of RAD using the following lemma, simple but useful. The proof is omitted.

\begin{lemma}\label{C4S1L1.1}
	Let $X$ and $Y$ be sequence classes and let $\lambda$ be a scalar sequence space such that $X(\mathbb{K})\stackrel{1}{=}Y(\mathbb{K})\stackrel{1}{=} \lambda_*$. If there is a Banach space $E$ such that $X(E)=Y(E)$  and the correspondence $u \in  \Pi_{\lambda_*^w;X}(\lambda,E) \mapsto  (u(e_j))_{j=1}^\infty \in X(E)$ is not surjective, then $Y$ is not ideal-representável.
\end{lemma}

\begin{corollary}
	The sequence class {\rm RAD} is not ideal-representable.
\end{corollary}
\begin{proof} Let $2 \leq n \in \mathbb{N}$ and $n < p < \infty$. The Banach space $\mathcal{L}(^n\ell_p)$ of all continuous $n$-linear forms on $\ell_p$ is reflexive by \cite[Proposition 4.1]{alencar}, so it contains no copy of $c_0$. Thus, ${\rm Rad}(\mathcal{L}(^n\ell_p)) = {\rm RAD}(\mathcal{L}(^n\ell_p))$ by \cite[Theorem 6.1]{vakhania}. Moreover, $\mathcal{L}(^n\ell_p)$ has no finite cotype (see \cite[Proposition 1.4]{botelho-1}), so the correspondence $u \in  \Pi_{\ell_2^w;{\rm Rad}}\left(\ell_{2},\mathcal{L}(^n\ell_p)\right) \mapsto  (u(e_j))_{j=1}^\infty \in {\rm Rad}\left(\mathcal{L}(^n\ell_p)\right)$ is not surjective by the proof of Corollary \ref{C4S1C1}. Using that ${\rm Rad}(\mathbb{K})\stackrel{1}{=}{\rm RAD}(\mathbb{K})\stackrel{1}{=}\ell_2$,  
the class RAD is not ideal-representable by Lemma \ref{C4S1L1.1}.
\end{proof}

In  strong contrast to the case $p = 1$ (cf. Proposition \ref{i9vq}), we have the following: 

\begin{proposition} For $1 < p < \infty$, the class $\ell_p(\cdot)$ is not ideal-representable.
\end{proposition}

\begin{proof} 
%

 Let $E$ be a non reflexive Banach space, say $E = c_0$. As subspaces of reflexive spaces are reflexive and quotients of reflexive spaces are reflexive, $E$ is not isomorphic to a quotient of a subspace of some $L_p$-space. By \cite[Theorem 3.1]{takahashi+okazaki}, there is an operator $u \in {\cal L}(\ell_{p^*}, E)$ for which the following equivalence does not hold:
	\begin{equation*} u\in \Pi_p(\ell_{p^*},E) \Longleftrightarrow \sum_{j=1}^\infty \|u(e_j)\|^p <+\infty. \end{equation*}
Since $(e_j)_{j=1}^\infty \in \ell_p^w(\ell_{p^*})$, the implication  $(\Longrightarrow)$ above holds for any operator in ${\cal L}(\ell_{p^*}, E)$, so the implication $(\Longleftarrow)$ is the one that fails for $u$. Therefore,  $(u(e_j))_{j=1}^\infty \in \ell_p(E)$ and $u \notin \Pi_p(\ell_{p^*},E)$. Assume that $\ell_p(\cdot)$ is ideal-representable. As $\ell_p(\mathbb{K}) = \ell_p = \ell_{p^*}^*$ and $\ell_p(\cdot)$ is linearly stable, the ideal $\Pi_{\ell_p^w;\ell_p(\cdot)}= \Pi_p$ represents $\ell_p(\cdot)$ by Theorem \ref{C4S1T1}, hence the operator
	$\Psi \colon \Pi_p(\ell_{p^*},E) \longrightarrow \ell_p(E)$, $\Psi(v) =  (v(e_j))_{j=1}^\infty,$ is surjective. Since $(u(e_j))_{j=1}^\infty \in \ell_p(E)$, there is $v \in \Pi_p(\ell_{p^*},E)$ such that 
	$$(u(e_j))_{j=1}^\infty = \Psi(v) = (v(e_j))_{j=1}^\infty. $$
So, $u(e_j) = v(e_j)$ for every $j$, from which follows that $u = v  \in \Pi_p(\ell_{p^*},E)$ because $(e_j)_{j=1}^\infty$ is a Schauder basis for  $\ell_{p^*}$. This contradiction proves that $\ell_p(\cdot)$ is not ideal-representable.
\end{proof}

\begin{remark}\rm Let $1< p<\infty$. In the theorem above we saw that the representation $ \Pi_p(\ell_{p^*},\cdot)  \approx   \ell_p(\cdot) $ is not true. Nevertheless, using the results of \cite[Section 25.10]{df} it can be proved that $ \Pi_p(\ell_{p^*},E)  \approx   \ell_p(E) $ whenever $E$ is a quotient of a subspace of some $L_p$-space. The particular case 
	$ \Pi_p(\ell_{p^*},\ell_p)  \approx   \ell_p(\ell_p)  $ can be computed directly  (see also \cite[Remark (v) p.\,26]{sinha+karn}).
\end{remark}

After so many negative applications of Theorem \ref{C4S1T1}, it is time for some positive application. We write $^*\ell_p:= \ell_{p^*}$ for $1 < p < \infty$ and $^*\ell_1:= c_0$.

\begin{example} \label{C4S1E5}\rm Let us see that, for each $1\leq p<\infty$, the sequence class $\ell_p^{\text{mid}}$ is ideal-representable. We shall do so by proving that $\ell_p^{\text{mid}}$ satisfies condition (iii) of Theorem \ref{C4S1T1}. Let $E$ be a Banach space, $(x_j)_{j=1}^\infty \in \ell_p^{\text{mid}}(E)$ and  $((\alpha_{j,i})_{i=1}^\infty)_{j=1}^\infty \in \ell_p^w(^*\ell_{p})$. 
For any $(\varphi_n)_{n=1}^\infty \in \ell_p^w(E^*)$,
	\begin{align*}
		\left(  \sum_{n=1}^\infty \sum_{j=1}^\infty \left|\varphi_n\left( \sum_{i=1}^\infty \alpha_{j,i} x_i\right)\right|^p  \right)^{1/p}&= \left(  \sum_{n=1}^\infty \sum_{j=1}^\infty \left| \sum_{i=1}^\infty \alpha_{j,i} \varphi_n(x_i)\right|^p  \right)^{1/p}\\
		&= \left(  \sum_{n=1}^\infty \sum_{j=1}^\infty \left| \langle (\varphi_n(x_i))_{i=1}^\infty,(\alpha_{j,i})_{i=1}^\infty  \rangle \right|^p  \right)^{1/p}\\
		&\leq \left(  \sum_{n=1}^\infty \|(\varphi_n(x_i))_{i=1}^\infty\|_p^p \cdot \|((\alpha_{j,i})_{i=1}^\infty)_{j=1}^\infty\|_{w,p}^p  \right)^{1/p}\\
		&= \|((\alpha_{j,i})_{i=1}^\infty)_{j=1}^\infty\|_{w,p} \cdot \left(  \sum_{n=1}^\infty \|(\varphi_n(x_i))_{i=1}^\infty\|_p^p   \right)^{1/p}\\
		&= \|((\alpha_{j,i})_{i=1}^\infty)_{j=1}^\infty\|_{w,p}  \left(  \sum_{n=1}^\infty \sum_{i=1}^\infty |\varphi_n(x_i)|^p   \right)^{1/p}\\
		&\leq \|((\alpha_{j,i})_{i=1}^\infty)_{j=1}^\infty\|_{w,p} \cdot \|(\varphi_n)_{n=1}^\infty\|_{w,p}\cdot \|(x_j)_{j=1}^\infty\|_{\text{mid},p},
	\end{align*}
where the last inequality follows by the definition of the norm $\|\cdot\|_{{\rm mid},p}$. This proves that $\left( \sum\limits_{i=1}^\infty \alpha_{j,i} x_i\right)_{j=1}^\infty  \in \ell_p^{\text{mid}}(E)$. Taking the supremum over $(\varphi_n)_{n=1}^\infty \in B_{\ell_p^w(E^*)}$, we get
	$$\left\|  \left( \sum_{i=1}^\infty \alpha_{j,i} x_i\right)_{j=1}^\infty  \right\|_{\text{mid},p}\leq  \|((\alpha_{j,i})_{i=1}^\infty)_{j=1}^\infty\|_{w,p} \cdot  \|(x_j)_{j=1}^\infty\|_{\text{mid},p}.$$
By Theorem \ref{C4S1T1}(iii), the class $\ell_p^{\text{mid}}$ is ideal-representable, and by (iv) it is represented by the ideal $\Pi_{\ell_p^w, \ell_p^{\rm mid}}$, which was studied in \cite{botelho+campos+santos, jamiljoed} under the name of ideal of weakly mid-$p$-summing operators.
\end{example}

\section{Further applications}
In this section we give a few more applications of our study on ideal-representable sequence classes, namely, we generalize results from \cite{junek+matos} and \cite{ariellama}.

The first application concerns a generalization of a result due to Junek and Matos \cite{junek+matos}. Let $1 \leq p < \infty$. On the one hand, the operators belonging to the ideal $\Pi_{\ell_p^w;\ell_p^u}$ were studied in \cite{gonzagut, pelc} and in \cite{junek+matos} they are called {\it unconditionally $p$-summing operators}. They denote  this ideal by ${\cal U}^p$. On the other hand, the operators belonging to the ideal $\Pi_{\ell_p^w;c_0(\cdot)}$ have been studied since the early 90's \cite{castillo90, castillo93}, for more recent developments, see \cite{ardakani, chendominguez}. The usual notation for this ideal is ${\cal C}^p$. In \cite[Theorem 1.7]{junek+matos}, it is proved that these two ideals coincide, that is,
\begin{equation}\label{9kme}\Pi_{\ell_p^w;\ell_p^u}= \mathcal{U}^p= \mathcal{C}^p= \Pi_{\ell_p^w;c_0(\cdot)}. \end{equation}

Next we generalize this result.

\begin{theorem}\label{C4S1P8}  Let $1 \leq p < \infty$. If $X$ is an ideal-representable sequence class with $X(\mathbb{K})=\ell_p$, then 
$\Pi_{X;\ell_p^u}=\Pi_{X;c_0(\cdot)}$.
\end{theorem}

\begin{proof} Let $E$ and $F$ be arbitrary Banach spaces. Our argument depends on the equality
\begin{equation}\label{C4S1Eq1}    		\Pi_{\ell_p^w;\ell_p^u}(^*\ell_{p},F)=\Pi_{\ell_p^w;c_0(\cdot)}(^*\ell_{p},F).
\end{equation}
Of course, this equality follows from (\ref{9kme}), but this is exactly the result we are generalizing. So, we will obtain (\ref{C4S1Eq1}) without using (\ref{9kme}). We shall do that by applying the following result that can be found in 
\cite{castillo93}: The following are equivalente for a bounded operator 
$u\colon E \longrightarrow F$:\\
(i) $u$ is $p$-convergent, that is,  $u\in \Pi_{\ell_p^w;c_0(\cdot)}(E,F)$.\\
(ii) For every $v \in \mathcal{L}(^*\ell_p,E)$, $u\circ v$ is a compact operator.  

Returning to the proof of (\ref{C4S1Eq1}), it is trivial that $\Pi_{\ell_p^w;\ell_p^u}(^*\ell_{p},F)\subseteq \Pi_{\ell_p^w;c_0(\cdot)}(^*\ell_{p},F)$ because $\ell_p^u(F) \subseteq c_0(F)$. Conversely, given $u\in \Pi_{\ell_p^w;c_0(\cdot)}(^*\ell_{p},F)$, consider the identity operator ${\rm id}_{^*\ell_p}$. 
The implication (i)$\Rightarrow$(ii) above gives $u=u\circ {\rm id} \in \mathcal{K}(^*\ell_{p},F)$. In Example \ref{C4S1E3} we saw that the class  $\ell_p^u$ is $(\mathcal{K},^*\ell_{p})$-representable, hence   $u \in \mathcal{K}(^*\ell_{p},F)= \Pi_{\ell_p^w;\ell_p^u}(^*\ell_{p},F)$. The equality (\ref{C4S1Eq1}) is proved. 

As to the desired equality,
the inclusion $\Pi_{X;\ell_p^u}\subseteq \Pi_{X;c_0(\cdot)}$ is trivial because $\ell_p^u(\cdot)\subseteq c_0(\cdot)$. For the reverse inclusion, let $u\in \Pi_{X;c_0(\cdot)}(E,F)$ and $(x_j)_{j=1}^\infty \in X(E)$ be given. Since the class $X$ is ideal-representable by assumption, Theorem \ref{C4S1T1} gives that $X$ is  $\Pi_{\ell_p^w;X}$-representable, hence we can take $v\in \Pi_{\ell_p^w;X}(^*\ell_{p},E)$ such that $(v(e_j))_{j=1}^\infty=(x_j)_{j=1}^\infty$. Using the ideal property of $\Pi_{\ell_p^w;X}$ and (\ref{C4S1Eq1}), we get
$$u\circ v \in \Pi_{\ell_p^w;c_0(\cdot)}(^*\ell_{p},F)\stackrel{(\ref{C4S1Eq1})}{=}\Pi_{\ell_p^w;\ell_p^u}(^*\ell_{p},F)\approx \ell_p^u(F),$$
from which it follows that $(u(x_j))_{j=1}^\infty=(u\circ v(e_j))_{j=1}^\infty\in \ell_p^u(F)$. This proves that $u\in \Pi_{X;\ell_p^u}(E,F)$.   	
\end{proof}

The next example shows that, besides of recovering the known coincidence (\ref{9kme}), the theorem above yields new related coincidences.

\begin{example}\label{C4S1E9}\rm
 Let $1 \leq p<\infty$. In Example \ref{C4S1E5} we saw, applying Theorem \ref{C4S1T1}, that the sequence  class $\ell_p^{\text{mid}}$ is ideal-representable. From Theorem \ref{C4S1P8} it follows that $$  \Pi_{\ell_p^{\text{mid}};\ell_p^u}= \Pi_{\ell_p^{\text{mid}};c_0(\cdot)}.$$   	
\end{example}

Proceeding to the second application, let $1 \leq p<\infty$. Bearing in mind the chain
$$\ell_p\langle\cdot \rangle \subseteq \ell_p(\cdot) \subseteq \ell_p^u \cap \ell_p^{\rm mid} \subseteq \ell_p^u \cup \ell_p^{\rm mid} \subseteq \ell_p^w, $$
it is natural to wonder if every linearly stable sequence class  $X$ such that $X(\mathbb{K}) = \ell_p$ lies between $\ell_p\langle\cdot \rangle$ and $\ell_p^w$. In \cite[Proposition 6.3]{ariellama} we gave a partial solution for $p > 1$ and $X$ enjoying certain properties. Now, with the theory of ideal-representable sequence classes, we are in the position to give a complete positive solution to the problem.

\begin{proposition}\label{C4S2P3}
 Let $1\leq p <\infty$. If $X$ is a linearly stable sequence class with  $X(\mathbb{K})=\ell_p$, then $\ell_p\langle\cdot\rangle \stackrel{1}{\hookrightarrow} X \stackrel{1}{\hookrightarrow} \ell_p^w$.
\end{proposition}
\begin{proof} The second inclusion is a particular case of Proposition \ref{C4S1P2}(v). As to the first inclusion, from Example \ref{C4S1E3} and Proposition \ref{i9vq} (remember that $\ell_1(\cdot) \stackrel{1}{=} \ell_1 \langle \cdot \rangle$)  we know that the sequence class $\ell_p\langle\cdot\rangle$ is represented by the Banach ideal  $\mathcal{N}$ of nuclear operators, which, in its turn, is the smallest Banach operator ideal \cite[Theorem 6.7.2]{pietsch}. Thus, given a Banach $E$ and a sequence $(x_j)_{j=1}^\infty \in \ell_p\langle E\rangle$, there exists an operator $u \in \mathcal{N}(^*\ell_p,E) \stackrel{1}{\hookrightarrow} \Pi_{\ell_p^w;X}(^*\ell_p,E) $ so that $(x_j)_{j=1}^\infty=(u(e_j))_{j=1}^\infty$. By Proposition \ref{C4S1P3} it follows that $(x_j)_{j=1}^\infty \in X(E)$ because $u\in \Pi_{\ell_p^w;X}(^*\ell_p,E)$. For the norm inequality, using that the norm of the operator $\Psi$ from Proposition \ref{C4S1P3} is not greater than 1, we have
$$ \|(x_j)_{j=1}^\infty\|_{X(E)}=\|(u(e_j))_{j=1}^\infty\|_{X(E)} = \|\Psi(u)\|_{X(E)}\leq \|u\|_{\Pi_{\ell_p^w;X}} \leq \|u\|_{\mathcal{N}}=\|(x_j)_{j=1}^\infty\|_{\ell_p\langle E\rangle}.$$
\end{proof}

%
%
%

\bigskip

\noindent Geraldo Botelho~~~~~~~~~~~~~~~~~~~~~~~~~~~~~~~~~~~~~~Ariel S. Santiago\\
Faculdade de Matem\'atica~~~~~~~~~~~~~~~~~~~~~~~~~~Departamento de Matemática\\
Universidade Federal de Uberl\^andia~~~~~~~~~~~~~Universidade Federal de Minas Gerais\\
38.400-902 -- Uberl\^andia -- Brazil~~~~~~~~~~~~~~~~~31.270-901 -- Belo Horizonte -- Brazil\\
e-mail: botelho@ufu.br~~~~~~~~~~~~~~~~~~~~~~~~~\,~~~~~e-mail: arielsantossantiago@hotmail.com

%

%
%
%


\begin{thebibliography}{99}\small



\vspace*{-0.5em}
\bibitem{achour} D. Achour, A. Alouani, P. Rueda, E. A. S\'anchez P\'erez, {\it Tensor characterizations of summing polynomials}, Mediterr. J. Math. {\bf 15} (2018), no. 3, Paper No. 127, 12 pp.

\vspace*{-0.5em}
\bibitem{achour2} D. Achour, A. Attallah, {\it Operator ideals generated by strongly Lorentz sequence spaces}, Adv. Oper. Theory {\bf 6} (2021), no. 2, Paper No. 35, 24 pp.


\vspace*{-0.5em}
\bibitem{alencar} R. Alencar, K. Floret,  \textit{Weak-strong continuity of multilinear mappings and the Pelczynsky-Pitt Theorem},  J. Math. Anal.  Appl. \textbf{ 206} (1997), 532--546.

\vspace*{-0.5em}

\bibitem{ardakani}  H. Ardakani, Kh. Taghavinejad, {\it The strong limited $p$-Schur property in Banach lattices}, Oper. Matrices {\bf 16} (2022), no. 3, 811–825.





\vspace*{-0.5em}
\bibitem{botelho-1} G. Botelho,  \textit{Cotype and absolutely summing multilinear mappings and homogeneous polynomials}, Proc. Roy. Irish Acad. Sect. A \textbf{ 97} (1997), 145--153.

\vspace*{-0.5em}

\bibitem {botelho+campos}
	G. Botelho, J. R. Campos, {\it On the transformation
			of vector-valued sequences by multilinear operators}, Monatsh. Math. {\bf 183} (2017), 415--435.

\vspace*{-0.5em}

\bibitem{botelho+campos+2} G. Botelho, J. R. Campos, {\it Duality theory for generalized summing linear operators}, Collect. Math. {\bf 74} (2023), 457-472.

\vspace*{-0.5em}

\bibitem{botelho+campos+santos} G. Botelho, J. R. Campos, J. Santos, {\it Operator ideals related to absolutely summing and Cohen strongly summing operators}, Pacific J. Math. {\bf 287} (2017), no. 1, 1–17.



\vspace*{-0.5em}

\bibitem{davidson}  G. Botelho and D. Freitas, {\it Summing multilinear operators by blocks: The isotropic and anisotropic cases}, J. Math. Anal. Appl. {\bf 490} (2020), no. 1, 124203, 21pp.

\vspace*{-0.5em}

\bibitem{ariellama} G. Botelho, A. S. Santiago, {\it Symmetric ideals of generalized summing multilinear operators}, Linear and Multilinear Algebra, DOI:
10.1080/03081087.2023.2300668, 2024.








\vspace*{-0.5em}

\bibitem{jamilsonrenato} J. R. Campos, R. Macedo, J. Santos, {\it An anisotropic summability and mixed sequences}, Rev. Mat. Complut. {\bf 37} (2024), 163-184.

\vspace*{-0.5em}

\bibitem{jamiljoed}J. R. Campos, J. Santos, {\it An anisotropic approach to mid summable sequences}, Colloq. Math. {\bf 161} (2020), no. 1, 35–49.

\vspace*{-0.5em}

\bibitem{carothers} N. L. Carothers, {\it A Short Course on Banach Space Theory}, London Mathematical Society, Student Texts 64, 2005.

\vspace*{-0.5em}

\bibitem{castillo90} J. M. F. Castillo, {\it $p$-converging and weakly-$p$-compact operators in $L_p$-spaces}, Proc. II. Congreso
An\'lisis Funcional Jarandilla, Extracta Math. {\bf 46} (1990), 46-54.


\vspace*{-0.5em}

\bibitem{castillo93} J. M. F. Castillo, F. S\'anchez, {\it Dunford-Pettis-like properties of continuous vector function spaces}, Rev. Mat. Univ. Complut. Madrid {\bf 6} (1993), no. 1, 43–59.

\vspace*{-0.5em}
\bibitem{chendominguez} D. Chen, J. A. Ch\'avez-Dom\'inguez,  L. Li,  \textit{$p$-converging operators and Dunford-Pettis property of order $p$}, J. Math. Anal. Appl. \textbf{ 461} (2018), 1053--1066.



\vspace*{-0.5em}

\bibitem{cohen} J. S. Cohen, {\it Absolutely $p$-summing, $p$-nuclear operators and their conjugates}, Math. Ann. {\bf 201} (1973), 177–200.

\vspace*{-0.5em}

\bibitem{df} A. Defant, K. Floret, {\it Tensor Norms and Operator Ideals}, North-Holland, 1993.


\vspace*{-0.5em}

\bibitem{diestel+jarchow+tonge} J. Diestel, H. Jarchow, A. 	Tonge, {\it Absolutely Summing Operators}, Cambridge University Press, 1995.









\vspace*{-0.5em}

\bibitem{fourie+rontgen} J. H. Fourie, I. M. R\"ontgen, {\it Banach space sequences and projective tensor products}, J. Math. Anal. Appl. 277 (2003),  629–644.

\vspace*{-0.5em}

\bibitem{gonzagut} M. Gonz\'alez, J. Guti\'errez, {\it Unconditionally converging polynomials on Banach spaces}, Math. Proc. Cambridge Philos. Soc. {\bf 117} (1995), 321-331.


\vspace*{-0.5em}
\bibitem{junek+matos} H. Junek, M. Matos,   \textit{On unconditionally $p$-summing and weakly $p$-convergent polynomials}, Arch. Math. (Basel) \textbf{70} (1998), 41--51.

\vspace*{-0.5em}
\bibitem{lt1} J. Lindenstrauss, L. Tzafriri, \textit{Classical Banach Spaces I and II}, Springer, 1996.





\vspace*{-0.5em}

\bibitem{pelc} A. Pe{\l}czy\'nski, {\it Banach spaces on which every unconditionally converging operator is weakly compact}, Bull. Acad. Pol. Sci. {\bf 10} (1962), 641--648.



\vspace*{-0.5em}



\bibitem{baweja} A. Philip, D. Baweja, {\it Ideals of mid $p$-summing operators: a tensor product approach}, Adv. Oper. Theory {\bf 8} (2023), no. 2, Paper No. 28, 18 pp.

\vspace*{-0.5em}
\bibitem{pietsch} A. Pietsch,  \textit{Operator Ideals}, North-Holland, 1980.

%


\vspace*{-0.5em}

\bibitem{baianosmedit} J. Ribeiro, F. Santos, {\it Generalized multiple summing multilinear operators on Banach spaces},  Mediterr. J. Math. {\bf 16} (2019), no. 5, Paper No. 108, 20 pp.

\vspace*{-0.5em}



\bibitem{baianosmfat}  J. Ribeiro, F. Santos, {\it Absolutely summing polynomials}, Methods Funct. Anal. Topology {\bf 27} (2021), no. 1, 74–87.



\vspace*{-0.5em}

\bibitem{ryan} R. A. Ryan, {\it Introduction to Tensor Products of Banach Spaces}, Springer, 2002.

\vspace*{-0.5em}
\bibitem{sinha+karn} D. Sinha, A. Karn, \textit{Compact operators whose adjoints
	factor through subspaces of $l_p$}, Studia Math.  \textbf{150} (2002), 17--33.

\vspace*{-0.5em}
\bibitem{takahashi+okazaki} Y. Takahasi, Y. Okazaki,  \textit{Characterizations
	of subspaces, quotients and subspaces of quotients of $L_p$}, Hokkaido Math. J. \textbf{15} (1986), 233--241.

\vspace*{-0.5em}
\bibitem{ricardo+tien} N. D. Tien, R. Vidal-V\'azquez, \textit{Convergence of the Rademacher Series in a Banach Space}, Vietnam J. Math. {\bf 26} (1998), 71-85.

\vspace*{-0.5em}

\bibitem{vakhania} N. N. Vakhania, V. I. Tarieladze, S. A. Chobanyan, {\it Probability Distributions on Banach Spaces}, D. Reidel Publishing Co., Dordrecht, 1987.




\end{thebibliography}
\end{document}